\documentclass[12pt]{amsart}

\usepackage{amsmath, amssymb, amscd, newlfont, bbm}
\usepackage[foot]{amsaddr}
\usepackage[final]{hyperref}
\hypersetup{
	colorlinks=true,        
	linkcolor={black},  
	citecolor={black},
	urlcolor={black}     
}
\usepackage{enumitem}
\usepackage{tikz-cd}
\usepackage{comment}
\usepackage{cite}

\newcommand{\spacedcdot}{{\,\cdot\,}}

\newcommand{\Q}{{\mathbb{Q}}}
\newcommand{\C}{{\mathbb{C}}}
\newcommand{\R}{{\mathbb{R}}}
\newcommand{\Z}{{\mathbb{Z}}}

\newcommand{\diag}{{\mathrm{diag}}}

\newcommand{\M}{{\mathrm{M}}}
\newcommand{\GL}{{\mathrm{GL}}}

\newcommand{\SO}{{\mathrm{SO}}}
\newcommand{\SL}{{\mathrm{SL}}}
\newcommand{\U}{{\mathrm{U}}}

\DeclareMathOperator{\Hol}{{\mathrm{Hol}}}

\DeclareMathOperator{\Ind}{{\mathrm{Ind}}}

\newcommand{\calD}{{\mathcal{D}}}
\newcommand{\calF}{{\mathcal{F}}}

\newcommand{\calH}{{\mathcal{H}}}

\newcommand{\calM}{{\mathcal{M}}}

\newcommand{\calS}{{\mathcal{S}}}

\usepackage[mathscr]{eucal}

\providecommand{\abs}[1]{\left\lvert#1\right\rvert}
\providecommand{\norm}[1]{\left\lVert#1\right\rVert}
\providecommand{\scal}[2]{\left<#1,#2\right>}

\newtheorem{theorem}{Theorem}[section]
\newtheorem{lemma}[theorem]{Lemma}
\newtheorem{proposition}[theorem]{Proposition}
\newtheorem{corollary}[theorem]{Corollary}

\theoremstyle{definition}

\theoremstyle{remark}
\newtheorem{remark}[theorem]{Remark}

\numberwithin{equation}{section}

\title{ Non-vanishing of vector-valued Poincar\'e series }

\author{Sonja \v Zunar}

\address{ Department of Mathematics,
	Faculty of Science,
	University of Zagreb,
	Bijeni\v cka 30,
	10000 Zagreb,
	Croatia}
\email{szunar@math.hr}

\subjclass[2010]{11F03, 11F99}
\thanks{The author acknowledges the Croatian Science Foundation grant IP-2018-01-3628.}
\keywords{Vector-valued modular forms, Poincar\'e series}

\begin{document}
\maketitle

\begin{abstract}
	We prove a vector-valued version of Mui\'c's integral non-vanishing criterion for Poincar\'e series on the upper half-plane $ \calH $. Moreover, we give an accompanying result on the construction of vector-valued modular forms in the form of Poincar\'e series. As an application of these results, we construct and study the non-vanishing of the classical and elliptic vector-valued Poincar\'e series.
\end{abstract}

\section{Introduction}

Vector-valued modular forms (VVMFs) have prominent applications in analytic number theory \cite{borcherds, borcherds_gross, eichler_zagier, selberg} and the theory of vertex operator algebras \cite{dong_li_mason,milas,miyamoto,zhu}. Although their usefulness was noticed already in the 1960s by A.\ Selberg \cite{selberg}, the theory of VVMFs was established in a systematic way only in the early 2000s by M.\ Knopp and G.\ Mason \cite{km_generalized, km_fourier, km_poincare}. Ever since, the theory has been steadily developing \cite{bantay1, bantay_gannon, gannon, marks, marks_mason, mason_2dim, mason_diff}, with many new results in the recent years \cite{bantay2, candelori_franc, franc_mason, franc_mason2, saber_sebbar}.

Let us fix a multiplier system $ v:\SL_2(\Z)\to\C_{\abs z=1} $ of weight $ k\in\R $. We recall that a VVMF of weight $ k $ for $ \SL_2(\Z) $ with respect to a representation $ \rho:\SL_2(\Z)\to\GL_p(\C) $ is a $ p $-tuple $ F=(F_1,\ldots,F_p) $ of holomorphic functions on $ \calH:=\C_{\Im(z)>0} $ that have suitable Fourier expansions, such that
\[ F\big|_k\gamma=\rho(\gamma)F,\qquad\gamma\in\SL_2(\Z), \]
where elements of $ \C^p $ are regarded as column-vectors, and $ \big|_k $ denotes the standard right action (depending on $ k $ and $ v $) of $ \SL_2(\Z) $ on the set $ \left(\C^p\right)^\calH $ of functions $ \calH\to\C^p $ (see \eqref{eq:006}).

Similarly as in the theory of classical modular forms (see, e.g., \cite[\S2.6]{miyake}), one of the simplest ways to construct a VVMF is to define it as the sum of a Poincar\'e series
\[ P_{\Lambda\backslash\SL_2(\Z),\rho}f:=\sum_{\gamma\in\Lambda\backslash\SL_2(\Z)}\rho(\gamma)^{-1}f\big|_k\gamma, \]
where $ \Lambda $ is a subgroup of $ \SL_2(\Z) $, and $ f $ is a suitable function $ \calH\to\C^p $ (cf.\ \cite[\S3]{km_poincare}). In the case when the constructed VVMF is cuspidal, the question whether it vanishes identically is non-trivial. In fact, it has no complete answer even in the scalar-valued case, although in that case it was recognized as interesting as early as H.\ Poincar\'e \cite[p.~249]{poincare}. In the scalar-valued case, most known approaches to addressing this question aim at individual families of Poincar\'e series and are based on estimates of their Fourier coefficients \cite{kohnen,lehner,rankin}. A different, more general approach was discovered by G.\ Mui\'c in 2009, when he proved an integral non-vanishing criterion for Poincar\'e series on unimodular locally compact Hausdorff groups \cite[Theorem 4.1]{muicMathAnn}, with applications in the theory of automorphic forms and automorphic representations (see, e.g., \cite{grobner}). As a corollary, he obtained a criterion for the non-vanishing of Poincar\'e series of integral weight on $ \calH $ \cite[Lemma 3.1]{muicIJNT}, applied it to several families of cusp forms \cite{muicJNT,muicIJNT,muicLFunk}, and we extended his results to the half-integral weight case \cite{zunarGlas,zunarManu,zunarRama}. 

In this paper, we prove a vector-valued version (Theorem \ref{thm:019}) of Mui\'c's integral non-vanishing criterion for Poincar\'e series on $ \calH $ and use it to study the non-vanishing of two families of cuspidal VVMFs, which we call, respectively, the classical and elliptic vector-valued Poincar\'e series, in analogy with their scalar-valued versions studied by H.\ Petersson \cite{petersson}. We also prove an accompanying result (Proposition \ref{prop:010}) on the construction of VVMFs in the form of vector-valued Poincar\'e series (VVPSs). Let us emphasize that our results apply only to the case when the representation $ \rho $ is unitary. Namely, the unitarity of $ \rho $ is indispensable in the computations, involving integrals and Poincar\'e series, that are at the heart of our proofs (see, e.g., the third equality in \eqref{eq:021}). On the other hand, a careful reader will notice that in the special case when $ p=1 $ and $ \rho $ is the trivial representation, we obtain results on scalar-valued Poincar\'e series on $ \calH $ of arbitrary real weight, while in the previous work on integral non-vanishing criteria only the integral and half-integral weights were considered.

The paper is organized as follows. After introducing some basic notation in Section \ref{sec:002}, in Section \ref{sec:003} we introduce vector spaces of VVMFs to be studied in this paper. In this, we essentially follow \cite{km_poincare}, the only difference being that whereas in \cite{km_poincare} only VVMFs for $ \SL_2(\Z) $ are considered, we work with VVMFs for a general subgroup $ \Gamma $ of finite index in $ \SL_2(\Z) $. As is well known, this slight generalization does not enlarge the class of considered VVMFs in a substantial way (see Lemma \ref{lem:005}), but it greatly simplifies the notation in subsequent sections.

In Section \ref{sec:004}, we prove a result on the construction of VVMFs in the form of VVPSs (Proposition \ref{prop:010}). In Section \ref{sec:005}, we prove our integral non-vanishing criterion for VVPSs (Theorem \ref{thm:019}). We end the paper by Sections \ref{sec:006} and \ref{sec:007}, in which we apply our results to the classical and elliptic VVPSs, respectively.\bigskip

I would like to thank Marcela Hanzer and Goran Mui\'c for their support and useful comments. The work on this paper was in part conducted while I was a visitor at the Faculty of Mathematics, University of Vienna. I would like to thank Harald Grobner and the University of Vienna for their hospitality.

\section{Basic notation}\label{sec:002}

Throughout the paper, let $ i:=\sqrt{-1}\in\C $ and
\[ z^k:=\abs{z}^k\,e^{ik\arg(z)},\qquad z\in\C^\times,\ \arg(z)\in\left]-\pi,\pi\right],\ k\in\R. \]

Let $ \calH:=\C_{\Im(z)>0} $. The group $ \SL_2(\R) $ acts on $ \calH\cup\R\cup\left\{\infty\right\} $ by linear fractional transformations:
\begin{equation}\label{eq:002}
g.\tau:=\frac{a\tau+b}{c\tau+d},\qquad g=\begin{pmatrix}a&b\\c&d\end{pmatrix}\in\SL_2(\R),\ \tau\in\calH\cup\R\cup\left\{\infty\right\}.
\end{equation}
Defining $ j:\SL_2(\R)\times\calH\to\C $, 
\[ j(g,\tau):=c\tau+d,\qquad g=\begin{pmatrix}a&b\\c&d\end{pmatrix}\in\SL_2(\R),\ \tau\in\calH, \]
we recall that
\begin{equation}\label{eq:007}
\Im(g.\tau)=\frac{\Im(\tau)}{\abs{j(g,\tau)}^2},\qquad g\in\SL_2(\R),\ \tau\in\calH.
\end{equation}

Throughout the paper, we fix $ p\in\Z_{>0} $, $ k\in\R $ and a unitary multiplier system $ v $ for $ \SL_2(\Z) $ of weight $ k $, i.e., a function $ v:\SL_2(\Z)\to\C_{\abs z=1} $ such that the function $ \mu:\SL_2(\Z)\times\calH\to\C $,
\[ \mu(\gamma,\tau):=v(\gamma)\,j(\gamma,\tau)^k, \]
is an automorphic factor, in the sense that
\[ \mu(\gamma_1\gamma_2,\tau)=\mu(\gamma_1,\gamma_2.\tau)\,\mu(\gamma_2,\tau),\qquad\gamma_1,\gamma_2\in\SL_2(\Z),\ \tau\in\calH. \]
Following \cite{km_poincare}, we impose on $ v $ the nontriviality condition
\begin{equation}\label{eq:001}
v(-I_2)=(-1)^{-k}
\end{equation}
and, writing $ T:=\begin{pmatrix}1&1\\0&1\end{pmatrix} $, define $ \kappa\in\left[0,1\right[ $ by the condition
\[ v(T)=e^{2\pi i\kappa}. \]

The group $ \SL_2(\Z) $ acts on the right on the space $ \left(\C^p\right)^\calH $ of functions $ \calH\to\C^p $ as follows:
\begin{equation}\label{eq:006}
\left(F\big|_k\gamma\right)(\tau):=v(\gamma)^{-1}\,j(\gamma,\tau)^{-k}\,F(\gamma.\tau),\qquad F\in\left(\C^p\right)^\calH,\ \gamma\in\SL_2(\Z),\ \tau\in\calH. 
\end{equation}
We note that due to the nontriviality condition \eqref{eq:001}, we have
\begin{equation}\label{eq:005}
F\big|_k(-I_2)=F,\qquad F\in\left(\C^p\right)^\calH.
\end{equation}

Next, we recall that the group 
\[ K:=\SO_2(\R)=\left\{\kappa_\theta:=\begin{pmatrix}\cos \theta&-\sin \theta\\\sin \theta&\cos \theta\end{pmatrix}:\theta\in\R\right\} \]
is a maximal compact subgroup of $ \SL_2(\R) $ and the stabilizer of $ i $ under the action \eqref{eq:002}. Let us denote
\[ n_x:=\begin{pmatrix}1&x\\0&1\end{pmatrix},\qquad a_y:=\begin{pmatrix}y^{\frac12}&0\\0&y^{-\frac12}\end{pmatrix},\qquad h_t:=\begin{pmatrix}e^t&0\\0&e^{-t}\end{pmatrix} \]
for $ x\in\R $, $ y\in\R_{>0} $ and $ t\in\R_{\geq0} $. By the Iwasawa (resp., Cartan) decomposition of $ \SL_2(\R) $, every $ g\in\SL_2(\R) $ can be written in the form
\[ g=n_xa_y\kappa_\theta=\kappa_{\theta_1}h_t\kappa_{\theta_2} \]
for some $ x\in\R $, $ y\in\R_{>0} $, $ t\in\R_{\geq0} $ and $\theta,\theta_1,\theta_2\in\R $, and then we have $ g.i=x+iy $. Denoting by $ \mathrm v $ the standard $ \SL_2(\R) $-invariant Radon measure on $ \calH $ given by $ d\mathrm v(x+iy):=\frac{dx\,dy}{y^2} $, we have the following Haar measure on $ \SL_2(\R) $:
\begin{align}
\int_{\SL_2(\R)}\varphi(g)\,dg&=\frac1{2\pi}\int_0^{2\pi}\int_\calH \varphi(n_xa_y\kappa_\theta)\,d\mathrm v(x+iy)\,d\theta\label{eq:017}\\
&=\frac1\pi\int_0^{2\pi}\int_0^{\infty}\int_0^{2\pi}\varphi(\kappa_{\theta_1}h_t\kappa_{\theta_2})\,\sinh(2t)\,d\theta_1\,dt\,d\theta_2,\quad\varphi\in C_c\left(\SL_2(\R)\right).\label{eq:039}
\end{align}
Moreover, for every discrete subgroup $ \Gamma $ of $ \SL_2(\R) $, we have an $ \SL_2(\R) $-invariant Radon measure on $ \Gamma\backslash\SL_2(\R) $ defined by the condition
\[ \int_{\Gamma\backslash\SL_2(\R)}\sum_{\gamma\in\Gamma}\varphi(\gamma g)\,dg=\int_{\SL_2(\R)}\varphi(g)\,dg,\qquad \varphi\in C_c\left(\SL_2(\R)\right), \]
or equivalently by the condition  
\begin{equation}\label{eq:045}
\int_{\Gamma\backslash\SL_2(\R)}\varphi(g)\,dg=\frac1{2\pi\abs{\Gamma\cap\left<-I_2\right>}}\int_0^{2\pi}\int_{\Gamma\backslash\calH}\varphi(n_xa_y\kappa_\theta)\,d\mathrm v(x+iy)\,d\theta
\end{equation}
for all $ \varphi\in C_c(\Gamma\backslash\SL_2(\R)) $.

Finally, let us mention that throughout the paper, for every $ n\in\Z_{>0} $ we regard the elements of $ \C^n $ as column-vectors. Moreover, we equip $ \C^n $ with the standard inner product
\[ \scal xy_{\C^n}:=\sum_{j=1}^nx_j\overline{y_j},\qquad x=(x_1,\ldots,x_n),y=(y_1,\ldots,y_n)\in\C^n, \]
and denote the induced norm on $ \mathbb C^n $ by $ \norm\spacedcdot $. We use the same notation for the Frobenius norm on the space $ M_n(\C) $ of complex square matrices of order $ n $:
\[ \norm X:=\sqrt{\sum_{r=1}^n\sum_{s=1}^n\abs{x_{r,s}}^2},\qquad X=(x_{r,s})_{r,s=1}^n\in M_n(\C). \]

\section{Vector-valued modular forms}\label{sec:003}

Let $ \Hol(\calH) $ denote the space of holomorphic functions $ \calH\to\C $. We define its subspace $ \calF(k) $ to consist of the functions $ f\in\Hol(\calH) $ with the following property: for every  $ \sigma\in\SL_2(\Z) $, the function $ f\big|_k\sigma $ has a Fourier expansion of the form
\[ \left(f\big|_k\sigma\right)(\tau)=\sum_{n=h_\sigma}^\infty a_n(\sigma)\,e^{2\pi i\frac n{N_\sigma}\tau},\qquad\tau\in\calH, \]
where $ h_\sigma\in\Z $, $ a_n(\sigma)\in\C $, and $ N_\sigma\in\Z_{>0} $. We note that the space $ \calF(k) $ is an $ \SL_2(\Z) $-submodule of $ \Hol(\calH) $ with respect to the action $ \big|_k $. Its submodule of functions $ f\in\calF(k) $ such that $ h_\sigma\geq0 $ (resp., $ h_\sigma>0 $) for all $ \sigma\in\SL_2(\Z) $ will be denoted by $ \calM(k) $ (resp., $ \calS(k) $).

Let $ \Gamma $ be a subgroup of finite index in $ \SL_2(\Z) $. We say that $ F=(F_1,\ldots,F_p)\in\calF(k)^p $
is a vector-valued modular form (VVMF) of weight $ k $ for $ \Gamma $ (with multiplier system $ v $) with respect to a representation $ \rho:\Gamma\to\GL_p(\C) $ if
\begin{equation}\label{eq:009}
F\big|_k\gamma=\rho(\gamma)F,\qquad \gamma\in\Gamma.
\end{equation}
We will denote the space of all such $ F $ by $ \calF(k,\rho,\Gamma) $. We also define its subspaces 
\[ \calM(k,\rho,\Gamma):=\calF(k,\rho,\Gamma)\cap\calM(k)^p \]
of entire VVMFs and
\[ \calS(k,\rho,\Gamma):=\calF(k,\rho,\Gamma)\cap\calS(k)^p \]
of cuspidal VVMFs.

It will prove useful to note that \eqref{eq:009} is equivalent to the condition
\begin{equation}\label{eq:014}
F\big|_{k,\rho}\gamma=F,\qquad\gamma\in\Gamma,
\end{equation}
where $ \big|_{k,\rho} $ is the right action of $ \Gamma $ on $ \left(\C^p\right)^\calH $ defined by
\[ F\big|_{k,\rho}\gamma:=\rho(\gamma)^{-1}F\big|_k\gamma,\qquad F\in\left(\C^p\right)^\calH,\ \gamma\in\Gamma. \]

We note that in most standard texts on VVMFs (see, e.g., \cite[\S2]{km_poincare}), only the spaces
\[ \calF(k,\rho):=\calF(k,\rho,\SL_2(\Z)),\ \calM(k,\rho):=\calM(k,\rho,\SL_2(\Z))\ \text{and}\ \calS(k,\rho):=\calS(k,\rho,\SL_2(\Z)) \] 
are studied. Our slightly more general definition of spaces of VVMFs serves only to facilitate the construction of VVMFs in Section \ref{sec:004} and does not enlarge the class of studied VVMFs in a substantial way. The latter observation is elementary and well-known (see, e.g., \cite[\S2]{selberg}, \cite[\S1--2]{borcherds} and \cite[\S1]{gannon}). For convenience of the reader, we provide its details in the following lemma.

\begin{lemma}\label{lem:005}
	Let $ \Gamma $ be a subgroup of finite index in $ \SL_2(\Z) $, and let $ \rho:\Gamma\to\GL_p(\C) $ be a representation.
	Denoting $ d:=\abs{\SL_2(\Z)/\Gamma} $, let us fix $ \gamma_1,\ldots,\gamma_d\in\SL_2(\Z) $ such that $ \SL_2(\Z)=\bigsqcup_{j=1}^d\Gamma\gamma_j $. Then, the rule
	\[ F\mapsto\left(F\big|_k\gamma_j\right)_{j=1}^d \]
	defines embeddings
	\[ \begin{aligned}
	\calF(k,\rho,\Gamma)&\hookrightarrow\calF(k,\rho_0),\\
	\calM(k,\rho,\Gamma)&\hookrightarrow\calM(k,\rho_0),\\
	\calS(k,\rho,\Gamma)&\hookrightarrow\calS(k,\rho_0),
	\end{aligned} \]
	where $ \rho_0:\SL_2(\Z)\to\GL_{pd}(\C) $ is a representation equivalent to the induced representation $ \Ind_\Gamma^{\SL_2(\Z)}(\rho) $ and defined as follows: for every $ \gamma\in\SL_2(\Z) $, defining a permutation $ \ell\in \mathfrak S_d $ by the rule
	\[ \Gamma\gamma_j\gamma^{-1}=\Gamma\gamma_{\ell(j)},\qquad j\in\left\{1,\ldots,d\right\}, \]
	we put
	\begin{equation}\label{eq:004}
	\rho_0(\gamma):=\begin{pmatrix}
	\delta_{1,\ell(1)}I_p&\ldots&\delta_{1,\ell(d)}I_p\\
	\vdots&&\vdots\\
	\delta_{d,\ell(1)}I_p&\ldots&\delta_{d,\ell(d)}I_p
	\end{pmatrix}\begin{pmatrix}\rho(\gamma_{\ell(1)}\gamma\gamma_1^{-1})&&\\&\ddots&\\&&\rho(\gamma_{\ell(d)}\gamma\gamma_d^{-1})\end{pmatrix}, 
	\end{equation}
	where $ \delta $ is the Kronecker delta. If $ \rho $ is unitary, then so is $ \rho_0 $.
\end{lemma}

\begin{proof}
	Let us prove the only two non-obvious parts of the claim:
	\begin{enumerate}[label=\textup{(\arabic*)},leftmargin=*,align=left]
		\item\label{enum:003:1} $ \left(F\big|_k\gamma_j\right)_{j=1}^d\big|_k\gamma=\rho_0(\gamma)\left(F\big|_k\gamma_j\right)_{j=1}^d $ for all $ \gamma\in\SL_2(\Z) $.
		\item\label{enum:003:2} $ \rho_0\cong\Ind_\Gamma^{\SL_2(\Z)}(\rho) $.\medskip
	\end{enumerate}
	
	\ref{enum:003:1} For $ \gamma\in\SL_2(\Z) $ and $ \ell $ as in the statement of the lemma, we have
	\[ \begin{aligned}
	\rho_0(\gamma)\left(F\big|_k\gamma_j\right)_{j=1}^d&\overset{\eqref{eq:004}}=\Big(\delta_{r,\ell(s)}I_p\Big)_{r,s=1}^d\,\diag\Big(\rho\left(\gamma_{\ell(j)}\gamma\gamma_j^{-1}\right)\Big)_{j=1}^d\,\left(F\big|_k\gamma_j\right)_{j=1}^d\\
	&\overset{\phantom{\eqref{eq:009}}}=\left(\rho \big(\gamma_j\gamma\gamma_{\ell^{-1}(j)}^{-1}\big)F\big|_k\gamma_{\ell^{-1}(j)}\right)_{j=1}^d\\
	&\overset{\eqref{eq:009}}=\left(F\big|_k\gamma_{j}\gamma\gamma_{\ell^{-1}(j)}^{-1}\big|_k\gamma_{\ell^{-1}(j)}\right)_{j=1}^d\\	
	&\overset{\phantom{\eqref{eq:009}}}=\left(F\big|_k\gamma_{j}\right)_{j=1}^d\big|_k\gamma.
	\end{aligned} \] 
	
	\ref{enum:003:2} Denoting $ \delta_j:=\gamma_j^{-1} $, we recall the following standard realization of $ \Ind_\Gamma^{\SL_2(\Z)}(\rho) $: for every $ j\in\left\{1,\ldots,d\right\} $, let $ \delta_j\C^p=\left\{\delta_ju:u\in\C^p\right\} $ be a complex vector space isomorphic to $ \C^p $ via $ \delta_ju\mapsto u $; then, $ \Ind_\Gamma^{\SL_2(\Z)}(\rho) $ can be defined as a representation of $ \SL_2(\Z) $ on $ \bigoplus_{j=1}^d\delta_j\C^p $ given by the formula
	\[ \begin{aligned}
	\left(\Ind_\Gamma^{\SL_2(\Z)}(\rho)\right)(\gamma)\left(\sum_{j=1}^d\delta_ju_j\right)&
	:=\sum_{j=1}^d\delta_{\ell(j)}\,\rho\left(\delta_{\ell(j)}^{-1}\gamma\delta_j\right)u_j\\
	&\,=\sum_{j=1}^d\delta_j\,\rho\left(\gamma_j\gamma\gamma_{\ell^{-1}(j)}^{-1}\right)u_{\ell^{-1}(j)}
	\end{aligned} \]
	for all $ \gamma\in\SL_2(\Z) $ and $ u_1,\ldots,u_d\in\C^p $. On the other hand, from \eqref{eq:004} we see that
	\[ \rho_0(\gamma)\left(u_j\right)_{j=1}^d=\left(\rho\left(\gamma_j\gamma\gamma_{\ell^{-1}(j)}^{-1}\right)u_{\ell^{-1}(j)}\right)_{j=1}^d,\qquad \left(u_j\right)_{j=1}^d\in\left(\C^p\right)^d. \]
	Thus, the rule
	\[ \left(u_j\right)_{j=1}^d\mapsto\sum_{j=1}^d\delta_ju_j \]
	defines an $ \SL_2(\Z) $-equivalence
	\[ \left(\rho_0,\C^{pd}\right)\cong\left(\Ind_\Gamma^{\SL_2(\Z)}(\rho),\bigoplus_{j=1}^d\delta_j\C^p\right).\qedhere \]
\end{proof}

The proof of the following lemma is identical to that of \cite[\S2, Proposition]{km_fourier}, so we omit it.

\begin{lemma}
	Let $ \Gamma $ be a subgroup of finite index in $ \SL_2(\Z) $. Let $ F=(F_1,\ldots,F_p)\in\calF(k)^p $ such that $ \C F_1+\ldots+\C F_p $ is a $ \Gamma $-submodule of $ \calF(k) $ with respect to the action $ \big|_k $. Then, there exists a representation $ \rho:\Gamma\to\GL_p(\C) $ such that $ F\in\calF(k,\rho,\Gamma) $.
\end{lemma}

By \eqref{eq:005}, for every subgroup $ \Gamma\not\ni-I_2 $ of finite index in $ \SL_2(\Z) $ and representation $ \rho:\Gamma\to\GL_p(\C) $, we have
\[ \calF(k,\rho,\Gamma)=\calF\left(k,\rho',\left<-I_2\right>\Gamma\right) \]
and analogously for the subspaces of entire (resp., cuspidal) VVMFs, where $ \rho' $ is the unique extension of $ \rho $ to $ \left<-I_2\right>\Gamma $ satisfying $ \rho'(-I_2)=I_p $. This shows that we may restrict our study of VVMFs to the case when $ -I_2\in\Gamma $.

\begin{lemma}\label{lem:040}
	Let $ \Gamma\ni-I_2 $ be a subgroup of finite index in $ \SL_2(\Z) $, and let $ \rho:\Gamma\to\GL_p(\C) $ be a representation. Suppose that there exists $ F=(F_1,\ldots,F_p)\in\calF(k,\rho,\Gamma) $ such that the functions $ F_1,\ldots,F_p $ are linearly independent. Then:
	\begin{enumerate}[label=\textup{(N\arabic*)},leftmargin=*,align=left]
		\item\label{enum:008:1} $ \rho(-I_2)=I_p $.
		\item\label{enum:008:2} Let $ \sigma\in\SL_2(\Z) $ and $ M\in\Z_{>0} $ such that $ \sigma T^M\sigma^{-1}\in\Gamma $. Then, there exists $ N\in\Z_{>0} $ such that
		\[ \left(e^{2\pi i\kappa M}\rho\left(\sigma T^M\sigma^{-1}\right)\right)^N=I_p. \]
	\end{enumerate}
\end{lemma}

\begin{proof}
	\ref{enum:008:1} We note that 
	\[ \rho(-I_2)F\overset{\eqref{eq:009}}=F\big|_k(-I_2)\overset{\eqref{eq:005}}=F, \]
	hence by the linear independence of $ F_1,\ldots,F_p $ it follows that $ \rho(-I_2)=I_p $.
	
	\ref{enum:008:2} Since $ F\in\calF(k)^p $, there exists $ N\in\Z_{>0} $ such that $ F\big|_k\sigma $ is $ N $-periodic. We have
	\[ \begin{aligned}
	\left(F\big|_k\sigma\right)(\tau)&\overset{\phantom{\eqref{eq:009}}}=\left(F\big|_k\sigma\right)(\tau+MN)\\
	&\overset{\phantom{\eqref{eq:009}}}=e^{2\pi i\kappa MN}\,\left(F\big|_k\sigma T^{MN}\right)(\tau)\\
	&\overset{\eqref{eq:009}}=\left(e^{2\pi i\kappa M}\rho\left(\sigma T^M\sigma^{-1}\right)\right)^N\,\left(F\big|_k\sigma\right)(\tau),\qquad\tau\in\calH,
	\end{aligned} \]
	so by the linear independence of $ F_1\big|_k\sigma,\ldots,F_p\big|_k\sigma $ it follows that
	\[ \left(e^{2\pi i\kappa M}\rho\left(\sigma T^M\sigma^{-1}\right)\right)^N=I_p.\qedhere \]
\end{proof}

From now until the end of this section, let $ \Gamma\ni-I_2 $ be a subgroup of finite index in $ \SL_2(\Z) $, and let $ \rho:\Gamma\to\GL_p(\C) $ be a representation. We say that $ \rho $ is a normal representation of $ \Gamma $ if it satisfies the conditions $ \ref{enum:008:1} $ and $ \ref{enum:008:2} $ of Lemma \ref{lem:040}.

Next, applying Lemma \ref{lem:005}, it follows from \cite[Lemma 2.4 and Theorem 2.5]{km_poincare} that the complex vector space $ \calM(k,\rho,\Gamma) $ is finite-dimensional for every $ k\in\R $ and is trivial if $ k\ll0 $. Moreover, by \cite[\S7]{km_poincare} we have the following lemma.

\begin{lemma}\label{lem:016}
	If $ \rho $ is unitary, then  $ \calM(k,\rho,\Gamma)=0 $ for $ k<0 $. 
\end{lemma}

The proof of the following lemma is analogous to \cite[proof of Theorem 2.1.5]{miyake}, and we leave it as an exercise to the reader.

\begin{lemma}\label{lem:041}
	Suppose that $ \rho $ is unitary. Let $ F\in\calS(k,\rho,\Gamma) $. Then, the function $ \calH\to\C $,
	\[ \tau\mapsto\norm{F(\tau)}\Im(\tau)^{\frac k2}, \]
	is $ \Gamma $-invariant and bounded on $ \calH $.
\end{lemma}

\begin{lemma}
	Suppose that $ \rho $ is unitary. Then, $ \calS(k,\rho,\Gamma) $ is a finite-dimensional Hilbert space under the Petersson inner product
	\begin{equation}\label{eq:029}
	\scal FG_{\calS(k,\rho,\Gamma)}:=\int_{\Gamma\backslash\calH}\scal{F(\tau)}{G(\tau)}_{\C^p}\,\Im(\tau)^k\,d\mathrm v(\tau),\qquad F,G\in\calS(k,\rho,\Gamma),
	\end{equation}
	and the embedding $ \calS(k,\rho,\Gamma)\hookrightarrow\calS(k,\rho_0) $ of Lemma \ref{lem:005} is an isometry.
\end{lemma}

\begin{proof}
	The space $ \calS(k,\rho,\Gamma) $ is finite-dimensional by our comments before Lemma \ref{lem:016}. One shows that the integrand in \eqref{eq:029} is $ \Gamma $-invariant as in \cite[proof of Lemma 5.1]{km_poincare}. Moreover, we have
	\[ \begin{aligned}
	\int_{\Gamma\backslash\calH}&\abs{\scal{F(\tau)}{G(\tau)}_{\C^p}}\,\Im(\tau)^k\,d\mathrm v(\tau)\\
	&\leq \mathrm v(\Gamma\backslash\calH)\,\left(\sup_{\tau\in\calH}\norm{F(\tau)}\Im(\tau)^{\frac k2}\right)\left(\sup_{\tau\in\calH}\norm{G(\tau)}\Im(\tau)^{\frac k2}\right)\overset{\text{Lem.\,\ref{lem:041}}}<\infty,
	\end{aligned} \]
	for all $ F,G\in\calS(k,\rho,\Gamma) $,	so the inner product \eqref{eq:029} is well-defined.
	
	Using the notation of Lemma \ref{lem:005}, the second claim of the lemma follows from the equality
	\[ \begin{aligned}
	&\scal{\left(F\big|_k\gamma_j\right)_{j=1}^d}{\left(G\big|_k\gamma_j\right)_{j=1}^d}_{\calS(k,\rho_0)}\\
	&\qquad=\int_{\SL_2(\Z)\backslash\calH}\sum_{j=1}^d\scal{\left(F\big|_k\gamma_j\right)(\tau)}{\left(G\big|_k\gamma_j\right)(\tau)}_{\C^p}\,\Im(\tau)^k\,d\mathrm v(\tau)\\
	&\qquad=\int_{\SL_2(\Z)\backslash\calH}\sum_{j=1}^d\scal{F(\gamma_j.\tau)}{G(\gamma_j.\tau)}_{\C^p}\,\Im(\gamma_j.\tau)^k\,d\mathrm v(\tau)\\
	&\qquad=\int_{\Gamma\backslash\calH}\scal{F(\tau)}{G(\tau)}_{\C^p}\,\Im(\tau)^k\,d\mathrm v(\tau)\\
	&\qquad=\scal FG_{\calS(k,\rho,\Gamma)},\qquad F,G\in\calS(k,\rho,\Gamma),	
	\end{aligned} \]
	where the second equality follows from \eqref{eq:006} and \eqref{eq:007} using the unitarity of $ v $.
\end{proof}

\section{Construction of vector-valued Poincar\'e series}\label{sec:004}

Let $ \Gamma $ be a subgroup of finite index in $ \SL_2(\Z) $, let $ \rho:\Gamma\to\GL_p(\C) $ be a representation, and let $ \Lambda $ be a subgroup of $ \Gamma $. The defining property \eqref{eq:014} of VVMFs suggests that, as in the classical theory (see, e.g., \cite[\S2.6]{miyake}), interesting elements of $ \calF(k,\rho,\Gamma) $ may be constructed in the form of a vector-valued Poincar\'e series (VVPS)
\[ P_{\Lambda\backslash\Gamma,\rho}f:=\sum_{\gamma\in\Lambda\backslash\Gamma}f\big|_{k,\rho}\gamma, \]
where $ f:\calH\to\C^p $ is a suitable function invariant under the $ \big|_{k,\rho} $-action of $ \Lambda $.
The following proposition, based on this idea, is a vector-valued version of \cite[Lemmas 3 and 5]{zunarRama} (see also \cite[first part of Lemma 2.3]{muicLFunk}).

\begin{proposition}\label{prop:010}
	Let $ \Gamma\ni-I_2 $ be a subgroup of finite index in $ \SL_2(\Z) $, and let $ \rho:\Gamma\to\GL_p(\C) $ be a unitary representation.
	Let $ \Lambda\ni-I_2 $ be a subgroup of $ \Gamma $.
	Let $ f:\calH\to\C^p $ be a measurable function with the following two properties:
	\begin{enumerate}[label=\textup{(f\arabic*)},leftmargin=*,align=left]
		\item\label{enum:011:1} $ f\big|_{k,\rho}\lambda=f $ for all $ \lambda\in\Lambda $.
		\item\label{enum:011:2} $ \int_{\Lambda\backslash\calH}\norm{f(\tau)}\,\Im(\tau)^{\frac k2}\,d\mathrm v(\tau)<\infty $.
	\end{enumerate}
	Then, we have the following:
	\begin{enumerate}[label=\textup{(\arabic*)},leftmargin=*,align=left]
		\item\label{prop:010:1} The Poincar\'e series 
		$ P_{\Lambda\backslash\Gamma,\rho}f $
		converges absolutely a.e.\ on $ \calH $, satisfies
		\begin{equation}\label{eq:013}
		\left(P_{\Lambda\backslash\Gamma,\rho}f\right)\big|_{k,\rho}\gamma=P_{\Lambda\backslash\Gamma,\rho}f,\qquad\gamma\in\Gamma, 
		\end{equation}
		and we have
		\begin{equation}\label{eq:015}
		\int_{\Gamma\backslash\calH}\norm{\left(P_{\Lambda\backslash\Gamma,\rho}f\right)(\tau)}\,\Im(\tau)^{\frac k2}\,d\mathrm v(\tau)\leq\int_{\Lambda\backslash\calH}\norm{f(\tau)}\,\Im(\tau)^{\frac k2}\,d\mathrm v(\tau).
		\end{equation}
		\item\label{prop:010:2} Suppose additionally that $ \rho $ is normal and that $ f\in\Hol(\calH)^p $. Then, the series $ P_{\Lambda\backslash\Gamma,\rho}f $ converges absolutely and uniformly on compact sets in $ \calH $ and defines an element of
		\[ \begin{cases}
		\calS(k,\rho,\Gamma),&\text{if }k\geq2\\
		\calM(k,\rho,\Gamma),&\text{if }0\leq k<2\\
		0,&\text{if }k<0.
		\end{cases} \]
	\end{enumerate}
\end{proposition}

\begin{proof}
	\ref{prop:010:1} One checks easily that the integral in \ref{enum:011:2} is well-defined, i.e., the integrand is $ \Lambda $-invariant, using \ref{enum:011:1}, \eqref{eq:007} and the unitarity of $ \rho $ and $ v $. The terms of the series $ P_{\Lambda\backslash\Gamma,\rho}f $ are also well-defined by \ref{enum:011:1}. All claims in \ref{prop:010:1} now easily follow from the estimate
	\begin{equation}\label{eq:012}
	\begin{aligned}
	\int_{\Gamma\backslash\calH}\sum_{\gamma\in\Lambda\backslash\Gamma}&\norm{\left(f\big|_{k,\rho}\gamma\right)(\tau)}\,\Im(\tau)^{\frac k2}\,d\mathrm v(\tau)\\
	&=\int_{\Gamma\backslash\calH}\sum_{\gamma\in\Lambda\backslash\Gamma}\norm{v(\gamma)^{-1}\rho(\gamma)^{-1}f(\gamma.\tau)}\,\abs{j(\gamma,\tau)}^{-k}\,\Im(\tau)^{\frac k2}\,d\mathrm v(\tau)\\
	&=\int_{\Gamma\backslash\calH}\sum_{\gamma\in\Lambda\backslash\Gamma}\norm{f(\gamma.\tau)}\,\Im(\gamma.\tau)^{\frac k2}\,d\mathrm v(\tau)\\
	&=\int_{\Lambda\backslash\calH}\norm{f(\tau)}\,\Im(\tau)^{\frac k2}\,d\mathrm v(\tau)\overset{\ref{enum:011:2}}<\infty,
	\end{aligned}
	\end{equation}
	where the second equality holds by \eqref{eq:007} and the unitarity of $ \rho $ and $ v $.
	
	\ref{prop:010:2} It follows easily from the estimate \eqref{eq:012} and \cite[Corollary 2.6.4]{miyake} that the series $ P_{\Lambda\backslash\Gamma,\rho}f $ converges absolutely and uniformly on compact sets in $ \calH $ and defines a function $ F\in\Hol(\calH)^p $. By \eqref{eq:013}, $ F $ satisfies \eqref{eq:014}.
	
	Next, let $ \sigma\in\SL_2(\Z) $. Denoting $ x:=\sigma.\infty $, by \cite[Theorem 1.5.4(2)]{miyake} there exists $ M\in\Z_{>0} $ such that $ \Gamma_x=\left<\pm\sigma T^M\sigma^{-1}\right> $. In particular, $ \sigma T^M\sigma^{-1}\in\Gamma $, hence by the unitarity of $ \rho $ and \ref{enum:008:2} there exist a unitary matrix $ U\in\U(p) $ and $ m_1,\ldots,m_p\in\left]0,1\right]\cap\Q $ such that
	\[ e^{2\pi i \kappa M}\,\rho\left(\sigma T^M\sigma^{-1}\right)=U^{-1}\diag(e^{2\pi im_1},\ldots,e^{2\pi im_p})U. \]
	By \eqref{eq:013}, 
	\[ F\big|_k\sigma T^M\sigma^{-1}=\rho\left(\sigma T^M\sigma^{-1}\right)F, \]
	hence 
	\[ UF\big|_k\sigma\big|_kT^M=e^{-2\pi i\kappa M}\,\diag(e^{2\pi i m_j})_{j=1}^p UF\big|_k\sigma, \]
	so for every $ j\in\left\{1,\ldots,p\right\} $ we have
	\[ \left((UF)_j\big|_k\sigma\right)(\tau+M)=e^{2\pi im_j}\,\left((UF)_j\big|_k\sigma\right)(\tau),\qquad\tau\in\calH, \]
	which implies that the (holomorphic) function $ \calH\to\C $,
	\[ \tau\mapsto e^{-2\pi i\frac{m_j}M\tau}\left((UF)_j\big|_k\sigma\right)(\tau), \]
	is $ M $-periodic, hence the function $ (UF)_j\big|_k\sigma $ has a Fourier expansion of the form
	\begin{equation}\label{eq:033}
	\left((UF)_j\big|_k\sigma\right)(\tau)=\sum_{n\in\Z}b_n(j)\,e^{2\pi i\frac{n+m_j}M\tau},\qquad\tau\in\calH,
	\end{equation}
	where $ b_n(j)\in\C $ are given by
	\begin{equation}\label{eq:044}
	b_n(j)=\frac1M\int_0^M\left((UF)_j\big|_k\sigma\right)(x+iy)\,e^{-2\pi i\frac{n+m_j}M(x+iy)}\,dx,\qquad y\in\R_{>0}.
	\end{equation}
	We have
	\begin{equation}\label{eq:043}
	\begin{aligned}
	\int_M^\infty&\abs{b_n(j)}\,e^{-2\pi\frac{n+m_j}My}\,y^{\frac k2-2}\,dy\\
	&\overset{\eqref{eq:044}}\leq\frac1M\int_{\left]0,M\right]\times\left]M,\infty\right[}\abs{\left((UF)_j\big|_k\sigma\right)(\tau)}\,\Im(\tau)^{\frac k2}\,d\mathrm v(\tau)\\
	&\underset{\eqref{eq:007}}{\overset{\eqref{eq:006}}=}\frac1M\int_{\left]0,M\right]\times\left]M,\infty\right[}\abs{(UF)_j(\sigma.\tau)}\,\Im(\sigma.\tau)^{\frac k2}\,d\mathrm v(\tau)\\
	&\overset{\phantom{\eqref{eq:007}}}=\frac1M\int_{\sigma.\left(\left]0,M\right]\times\left]M,\infty\right[\right)}\abs{(UF)_j(\tau)}\,\Im(\tau)^{\frac k2}\,d\mathrm v(\tau)\\
	&\overset{\phantom{\eqref{eq:007}}}\leq\frac1M\int_{\Gamma\backslash\calH}\norm{(UF)(\tau)}\,\Im(\tau)^{\frac k2}\,d\mathrm v(\tau)\\
	&\overset{\phantom{\eqref{eq:007}}}=\frac1M\int_{\Gamma\backslash\calH}\norm{F(\tau)}\,\Im(\tau)^{\frac k2}\,d\mathrm v(\tau)\\
	&\overset{\eqref{eq:015}}<\infty,
	\end{aligned} 
	\end{equation}
	where the second inequality holds because by \cite[Corollary 1.7.5]{miyake} no two different points of $ \sigma.\left(\left]0,M\right]\times\left]M,\infty\right[\right) $ are mutually $ \Gamma $-equivalent.
	The estimate \eqref{eq:043} implies that $ b_n(j)=0 $ if either $ n+m_j=0 $ and $ k\geq2 $ or $ n+m_j<0 $. This means that the functions $ (UF)_j $ satisfy
	\[ (UF)_j\in\begin{cases}
	\calS(k),&\text{if }k\geq2\\
	\calM(k),&\text{if }k<2,
	\end{cases}\qquad j\in\left\{1,\ldots,p\right\}, \]
	so the same holds for their linear combinations $ F_j $. It follows that
	\[ F\in\begin{cases}
	\calS(k,\rho,\Gamma),&\text{if }k\geq2\\
	\calM(k,\rho,\Gamma),&\text{if }k<2.
	\end{cases} \]
	Finally, the claim in the case when $ k<0 $ follows by Lemma \ref{lem:016}.
\end{proof}

\section{A non-vanishing criterion for vector-valued Poincar\'e series}\label{sec:005}

Let $ \Gamma\ni-I_2 $ be a subgroup of finite index in $ \SL_2(\Z) $, and let $ \rho:\Gamma\to\GL_p(\C) $ be a unitary representation.
Moreover, let $ \Lambda\ni-I_2 $ be a subgroup of $ \Gamma $.

We start this section with a technical lemma.

\begin{lemma}\label{lem:025}
	Let $ f:\calH\to\C^p $ be a measurable function satisfying \ref{enum:011:1}, and let $ A $ be a Borel-measurable subset of $ \calH $. Then,
	\[ \int_{\Lambda\backslash\Lambda.A}\norm{f(\tau)}\,\Im(\tau)^{\frac k2}\,d\mathrm v(\tau)=2\int_{\Lambda\backslash\underline{\Lambda.A}}\norm{f(g.i)}\,\abs{j(g,i)}^{-k}\,dg, \]
	where we use the notation
	\begin{equation}\label{eq:023}
	\underline S:=\left\{n_xa_y:x+iy\in S\right\}K=\left\{g\in\SL_2(\R):g.i\in S\right\},\qquad S\subseteq\calH.
	\end{equation}
\end{lemma}

\begin{proof}
	Denoting by $ \mathbbm 1_{\Lambda.A} $ (resp., $ \mathbbm1_{\underline{\Lambda.A}} $) the characteristic function of $ \Lambda.A $ (resp., $ \underline{\Lambda.A} $) in $ \calH $ (resp., $ \SL_2(\R) $), we have
	\begin{align*}
	&\int_{\Lambda\backslash\Lambda.A}\norm{f(\tau)}\,\Im(\tau)^{\frac k2}\,d\mathrm v(\tau)\\
	&\overset{\phantom{\eqref{eq:045}}}=\frac1{2\pi}\int_0^{2\pi}\int_{\Lambda\backslash\calH}\norm{f(x+iy)}\,y^{\frac k2}\,\mathbbm1_{\Lambda.A}(x+iy)\,d\mathrm v(x+iy)\,d\theta\\
	&\overset{\phantom{\eqref{eq:045}}}=\frac1{2\pi}\int_0^{2\pi}\int_{\Lambda\backslash\calH}\norm{f(n_xa_y\kappa_\theta.i)}\,\abs{j(n_xa_y\kappa_\theta,i)}^{-k}\,\mathbbm1_{\underline{\Lambda.A}}(n_xa_y\kappa_\theta)\,d\mathrm v(x+iy)\,d\theta\\
	&\overset{\eqref{eq:045}}=2\int_{\Lambda\backslash\underline{\Lambda.A}}\norm{f(g.i)}\,\abs{j(g,i)}^{-k}\,dg.\qedhere
	\end{align*}
\end{proof}

The following theorem may be regarded as a vector-valued version of the integral non-vanishing criterion \cite[Lemma 3.1]{muicIJNT} for Poincar\'e series of integral weight on $ \calH $ (see also \cite[Theorem 2]{zunarRama} for the half-integral weight version, and \cite[Theorem 4.1]{muicMathAnn} for the original version of the criterion, in which Poincar\'e series on unimodular locally compact Hausdorff groups are considered).

\begin{theorem}\label{thm:019}
	Let $ \Gamma\ni-I_2 $ be a subgroup of finite index in $ \SL_2(\Z) $, and let $ \rho:\Gamma\to\GL_p(\C) $ be a unitary representation.
	Let $ \Lambda\ni-I_2 $ be a subgroup of $ \Gamma $, and let $ f:\calH\to\C^p $ be a measurable function with the following properties:
	\begin{enumerate}[label=\textup{(f\arabic*)},leftmargin=*,align=left]
		\item\label{enum:018:1} $ f\big|_{k,\rho}\lambda=f $ for all $ \lambda\in\Lambda $.
		\item[\textup{(f2')}] The series $ P_{\Lambda\backslash\Gamma,\rho}f $ converges absolutely a.e.\ on $ \calH $.\label{enum:018:2}
	\end{enumerate}
	Then, we have that
	\[ \int_{\Gamma\backslash\calH}\norm{\left(P_{\Lambda\backslash\Gamma,\rho}f\right)(\tau)}\,\Im(\tau)^{\frac k2}\,d\mathrm v(\tau)>0 \] 
	if one of the following holds:
	\begin{enumerate}[label=\textup{(\roman*)},leftmargin=*,align=left]
		\item\label{thm:019:1} There exists a Borel-measurable set $ A\subseteq\calH $ with the following properties:
		\begin{enumerate}[label=\textup{(A\arabic*)},leftmargin=*,align=left]
			\item\label{thm:019:1:1} No two points of $ A $ are mutually $ \Gamma $-equivalent.
			\item\label{thm:019:1:2} Denoting $ (\Lambda.A)^c:=\calH\setminus\Lambda.A $, we have
			\[ \int_{\Lambda\backslash\Lambda.A}\norm{f(\tau)}\,\Im(\tau)^{\frac k2}\,d\mathrm v(\tau)>\int_{\Lambda\backslash(\Lambda.A)^c}\norm{f(\tau)}\,\Im(\tau)^{\frac k2}\,d\mathrm v(\tau). \]
		\end{enumerate}
		\item\label{thm:019:2} There exists a Borel-measurable set $ C\subseteq\SL_2(\R) $ with the following properties:
		\begin{enumerate}[label=\textup{(C\arabic*)},leftmargin=*,align=left]
			\item\label{thm:019:2:1} $ CK=C $.
			\item\label{thm:019:2:2} $ CC^{-1}\cap\Gamma\subseteq\left<-I_2\right> $.
			\item\label{thm:019:2:3} Denoting $ (\Lambda C)^c:=\SL_2(\R)\setminus\Lambda C $, we have
			\[ \int_{\Lambda\backslash\Lambda C}\norm{f(g.i)}\,\abs{j(g,i)}^{-k}\,dg>\int_{\Lambda\backslash(\Lambda C)^c}\norm{f(g.i)}\,\abs{j(g,i)}^{-k}\,dg. \]
		\end{enumerate}
	\end{enumerate}
\end{theorem}

\begin{remark}
	By Proposition \ref{prop:010}\ref{prop:010:1}, Theorem \ref{thm:019} remains true if we replace the property \textup{(f2')} in it by \ref{enum:011:2}.
\end{remark}

\begin{proof}[Proof of Theorem \ref{thm:019}]
	Suppose that \ref{thm:019:1} holds. First, we recall that by \cite[Theorem 1.7.8]{miyake}, the set of elliptic points for $ \Gamma $ in $ \calH $ is countable, hence of measure zero. Next, we note that if $ \tau\in\calH $ is not an elliptic point for $ \Gamma $, i.e., if $ \Gamma_\tau=\left<-I_2\right> $, then
	\begin{equation}\label{eq:020}
	\#\left\{\gamma\in\Lambda\backslash\Gamma:\mathbbm1_{\Lambda.A}(\gamma.\tau)\neq0\right\}\leq1.
	\end{equation}
	Namely, if $ \gamma.\tau,\gamma'.\tau\in\Lambda.A $ for some $ \gamma,\gamma'\in\Gamma $, then there exist $ \lambda,\lambda'\in\Lambda $ such that $ \lambda\gamma.\tau,\lambda'\gamma'.\tau\in A $, hence by \ref{thm:019:1:1} we have $ \lambda\gamma.\tau=\lambda'\gamma'.\tau $, which by the non-ellipticity of $ \tau $ implies that $ \lambda'\gamma'\in\left\{\pm\lambda\gamma\right\} $, hence $ \Lambda\gamma'=\Lambda\gamma $.
	
	Denoting by $ \mathbbm1_S $ the characteristic function of a set $ S\subseteq\calH $, we have
	\begin{equation}\label{eq:021}
	\begin{aligned}
	\int_{\Gamma\backslash\calH}&\norm{\left(P_{\Lambda\backslash\Gamma,\rho}\left(\mathbbm1_{\Lambda.A}\,f\right)\right)(\tau)}\,\Im(\tau)^{\frac k2}\,d\mathrm v(\tau)\\
	&\overset{\phantom{\eqref{eq:020}}}=\int_{\Gamma\backslash\calH}\norm{\sum_{\gamma\in\Lambda\backslash\Gamma}\mathbbm1_{\Lambda.A}(\gamma.\tau)\,\rho(\gamma)^{-1}\left(f\big|_k\gamma\right)(\tau)}\,\Im(\tau)^{\frac k2}\,d\mathrm v(\tau)\\
	&\overset{\eqref{eq:020}}=\int_{\Gamma\backslash\calH}\sum_{\gamma\in\Lambda\backslash\Gamma}\mathbbm1_{\Lambda.A}(\gamma.\tau)\,\norm{\rho(\gamma)^{-1}\left(f\big|_k\gamma\right)(\tau)}\,\Im(\tau)^{\frac k2}\,d\mathrm v(\tau)\\
	&\overset{\phantom{\eqref{eq:020}}}=\int_{\Gamma\backslash\calH}\sum_{\gamma\in\Lambda\backslash\Gamma}\mathbbm1_{\Lambda.A}(\gamma.\tau)\,\norm{f(\gamma.\tau)}\,\Im(\gamma.\tau)^{\frac k2}\,d\mathrm v(\tau)\\
	&\overset{\phantom{\eqref{eq:020}}}=\int_{\Lambda\backslash\calH}\mathbbm1_{\Lambda.A}(\tau)\,\norm{f(\tau)}\,\Im(\tau)^{\frac k2}\,d\mathrm v(\tau)\\
	&\overset{\phantom{\eqref{eq:020}}}=\int_{\Lambda\backslash\Lambda.A}\norm{f(\tau)}\,\Im(\tau)^{\frac k2}\,d\mathrm v(\tau),
	\end{aligned}
	\end{equation}
	where the third equality holds by \eqref{eq:007}, \eqref{eq:006} and the unitarity of $ \rho $ and $ v $.
	
	On the other hand, we have
	\begin{equation}\label{eq:022}
	\begin{aligned}
	\int_{\Gamma\backslash\calH}&\norm{\left(P_{\Lambda\backslash\Gamma,\rho}\left(\mathbbm1_{(\Lambda.A)^c}\,f\right)\right)(\tau)}\,\Im(\tau)^{\frac k2}\,d\mathrm v(\tau)\\
	&\leq\int_{\Gamma\backslash\calH}\sum_{\gamma\in\Lambda\backslash\Gamma}\mathbbm1_{(\Lambda.A)^c}(\gamma.\tau)\,\norm{\rho(\gamma)^{-1}\left(f\big|_k\gamma\right)(\tau)}\,\Im(\tau)^{\frac k2}\,d\mathrm v(\tau)\\
	&=\int_{\Gamma\backslash\calH}\sum_{\gamma\in\Lambda\backslash\Gamma}\mathbbm1_{(\Lambda.A)^c}(\gamma.\tau)\,\norm{f(\gamma.\tau)}\,\Im(\gamma.\tau)^{\frac k2}\,d\mathrm v(\tau)\\
	&=\int_{\Lambda\backslash\left(\Lambda.A\right)^c}\norm{f(\tau)}\,\Im(\tau)^{\frac k2}\,d\mathrm v(\tau).
	\end{aligned}
	\end{equation}
	
	Thus, 
	\[ \begin{aligned}
	&\int_{\Gamma\backslash\calH}\norm{\left(P_{\Lambda\backslash\Gamma,\rho}f\right)(\tau)}\Im(\tau)^{\frac k2}\,d\mathrm v(\tau)\\
	&\underset{\phantom{\eqref{eq:022}}}\geq\int_{\Gamma\backslash\calH}\norm{\left(P_{\Lambda\backslash\Gamma,\rho}\left(\mathbbm1_{\Lambda.A}\,f\right)\right)(\tau)}\,\Im(\tau)^{\frac k2}\,d\mathrm v(\tau)\\
	&\phantom{\underset{\eqref{eq:022}}\geq}-\int_{\Gamma\backslash\calH}\norm{\left(P_{\Lambda\backslash\Gamma,\rho}\left(\mathbbm1_{(\Lambda.A)^c}\,f\right)\right)(\tau)}\,\Im(\tau)^{\frac k2}\,d\mathrm v(\tau)\\
	&\underset{\eqref{eq:022}}{\overset{\eqref{eq:021}}\geq}\int_{\Lambda\backslash\Lambda.A}\norm{f(\tau)}\,\Im(\tau)^{\frac k2}\,d\mathrm v(\tau)
	-\int_{\Lambda\backslash(\Lambda.A)^c}\norm{f(\tau)}\,\Im(\tau)^{\frac k2}\,d\mathrm v(\tau)\\
	&\underset{\phantom{\eqref{eq:022}}}{\overset{\ref{thm:019:1:2}}>}0.
	\end{aligned}\medskip \]
	
	Next, suppose that \ref{thm:019:2} holds. To finish the proof of the theorem, it suffices to prove that the set
	\begin{equation}\label{eq:024}
	A:=C.i\overset{\ref{thm:019:2:1}}=\left\{x+iy:n_xa_yK\subseteq C\right\} 
	\end{equation}
	has the properties \ref{thm:019:1:1} and \ref{thm:019:1:2}.
	
	\ref{thm:019:1:1} Suppose that $ \gamma.(x+iy)=x'+iy' $ for some $ \gamma\in\Gamma $ and $ x+iy,x'+iy'\in A $. Equivalently, $ \gamma n_xa_y.i=n_{x'}a_{y'}.i $, i.e., $ a_{y'}^{-1}n_{x'}^{-1}\gamma n_xa_y.i=i $, hence $ a_{y'}^{-1}n_{x'}^{-1}\gamma n_xa_y\in K $, so 
	\[ \gamma\in\left(n_{x'}a_{y'}K\right)\left(n_xa_y\right)^{-1}\cap\Gamma\overset{\eqref{eq:024}}\subseteq CC^{-1}\cap\Gamma\overset{\ref{thm:019:2:2}}\subseteq\left<-I_2\right>, \]
	which implies that $ x+iy=x'+iy' $.
	
	\ref{thm:019:1:2} Using the notation \eqref{eq:023}, by \ref{thm:019:2:1} and \eqref{eq:024} we have that $ C=\underline A $, $ \Lambda C=\underline{\Lambda.A} $, and $ (\Lambda C)^c=\underline{(\Lambda.A)^c} $, so \ref{thm:019:1:2} follows from \ref{thm:019:2:3} by applying Lemma \ref{lem:025}.
\end{proof}

\section{Classical vector-valued Poincar\'e series}\label{sec:006}

As a first example application of our results, in this section we construct and study the non-vanishing of the cuspidal VVMFs that are vector-valued analogues of the classical Poincar\'e series (for details on the latter cusp forms, see, e.g., \cite[Theorems 2.6.9(1) and 2.6.10]{miyake}). We note that in the case when $ \Gamma=\SL_2(\Z) $, these VVMFs have already been studied in \cite[\S3]{km_poincare}.

We will need the following lemma.

\begin{lemma}\label{lem:037}
	Let $ \Gamma\ni-I_2 $ be a subgroup of finite index in $ \SL_2(\Z) $, let $ \rho:\Gamma\to\GL_p(\C) $ be a unitary representation, and let $ \Lambda\ni-I_2 $ be a subgroup of $ \Gamma $. Let $ f:\calH\to\C^p $ be a measurable function satisfying \ref{enum:011:1} and \ref{enum:011:2}, such that $ P_{\Lambda\backslash\Gamma,\rho}f\in\calS(k,\rho,\Gamma) $. Then,
	\begin{equation}\label{eq:031}
	\scal F{P_{\Lambda\backslash\Gamma,\rho}f}_{\calS(k,\rho,\Gamma)}=\int_{\Lambda\backslash\calH}\scal{F(\tau)}{f(\tau)}_{\C^p}\,\Im(\tau)^k\,d\mathrm v(\tau),\qquad F\in\calS(k,\rho,\Gamma).
	\end{equation}
\end{lemma}

\begin{proof}
	We have
	\[ \begin{aligned}
	&\scal F{P_{\Lambda\backslash\Gamma,\rho}f}_{\calS(k,\rho,\Gamma)}
	\overset{\eqref{eq:029}}=\int_{\Gamma\backslash\calH}\scal{F(\tau)}{\left(P_{\Lambda\backslash\Gamma,\rho}f\right)(\tau)}_{\C^p}\,\Im(\tau)^k\,d\mathrm v(\tau)\\
	&\quad\overset{\eqref{eq:014}}=\int_{\Gamma\backslash\calH}\sum_{\gamma\in\Lambda\backslash\Gamma}\scal{\rho(\gamma)^{-1}\left(F\big|_k\gamma\right)(\tau)}{\rho(\gamma)^{-1}\left(f\big|_k\gamma\right)(\tau)}_{\C^p}\,\Im(\tau)^k\,d\mathrm v(\tau)\\
	&\quad\overset{\phantom{\eqref{eq:007}}}=\int_{\Gamma\backslash\calH}\sum_{\gamma\in\Lambda\backslash\Gamma}\scal{F(\gamma.\tau)}{f(\gamma.\tau)}_{\C^p}\,\Im(\gamma.\tau)^k\,d\mathrm v(\tau)\\
	&\quad\overset{\phantom{\eqref{eq:029}}}=\int_{\Lambda\backslash\calH}\scal{F(\tau)}{f(\tau)}_{\C^p}\,\Im(\tau)^k\, d\mathrm v(\tau),\qquad F\in\calS(k,\rho,\Gamma),
	\end{aligned} \]
	where in the third equality we used \eqref{eq:007}, \eqref{eq:006} and the unitarity of $ \rho $ and $ v $.
\end{proof}

Let $ \Gamma\ni-I_2 $ be a subgroup of finite index in $ \SL_2(\Z) $, let $ \rho:\Gamma\to\GL_p(\C) $ be a normal unitary representation, and let $ M\in\Z_{>0} $ such that $ \Gamma_\infty=\left<\pm T^M\right> $ (see \cite[Theorem 1.5.4(2)]{miyake}). By the unitarity of $ \rho $ and \ref{enum:008:2}, there exist $ U\in\U(p) $ and $ m_1,\ldots,m_p\in\left]0,1\right]\cap\Q $ such that
\begin{equation}\label{eq:026}
\rho\left(T^M\right)=e^{-2\pi i\kappa M}\,U^{-1}\,\diag(e^{2\pi im_1},\ldots,e^{2\pi im_p})\,U.
\end{equation}

\begin{proposition}\label{prop:030}
	Let $ k\in\R_{>2} $, $ \nu\in\Z_{\geq0} $, and $ j\in\left\{1,\ldots,p\right\} $. Denoting by $ e_j $ the $ j $th vector of the canonical basis for $ \C^p $, we have the following:
	\begin{enumerate}[label=\textup{(\arabic*)},leftmargin=*,align=left]
		\item\label{enum:028:1} The Poincar\'e series
		\[ \Psi_{k,\rho,\Gamma,\nu,U,j}:=P_{\Gamma_\infty\backslash\Gamma,\rho}\left(e^{2\pi i\frac{\nu+m_j}M\spacedcdot}\,U^{-1}e_j\right) \]
		converges absolutely and uniformly on compact sets in $ \calH $ and defines an element of $ \calS(k,\rho,\Gamma) $.
		\item\label{enum:028:2} For every $ F\in\calS(k,\rho,\Gamma) $, we have
		\begin{equation}\label{eq:027}
		\scal F{\Psi_{k,\rho,\Gamma,\nu,U,j}}_{\calS(k,\rho,\Gamma)}=b_\nu(j)\,\frac{M^k\,\Gamma(k-1)}{\left(4\pi(\nu+m_j)\right)^{k-1}},
		\end{equation}
		where $ b_\nu(j)\in\C $ are coefficients in the Fourier expansion
		\begin{equation}\label{eq:032}
		(UF)_j(\tau)=\sum_{n=0}^\infty b_n(j)\,e^{2\pi i\frac{n+m_j}M\tau},\qquad\tau\in\calH,
		\end{equation}
		and $ \Gamma $ on the right-hand side of \eqref{eq:027} denotes the gamma function $ \Gamma(s):=\int_0^\infty t^{s-1}\,e^{-t}\,dt,\ \Re(s)>0 $.
	\end{enumerate}
\end{proposition}

\begin{proof}
	\ref{enum:028:1} By Proposition \ref{prop:010}, it suffices to prove that the function $ f:\calH\to\C^p $,
	\begin{equation}\label{eq:025}
	f(\tau):=e^{2\pi i\frac{\nu+m_j}M\tau}\,U^{-1}e_j, 
	\end{equation}
	satisfies \ref{enum:011:1} and \ref{enum:011:2} with $ \Lambda=\Gamma_\infty $. The property \ref{enum:011:1} is satisfied by \ref{enum:008:1}, \eqref{eq:005} and the equality
	\[ \left(f\big|_{k,\rho}T^M\right)(\tau)=e^{-2\pi i\kappa M}\,\rho\left(T^{-M}\right)f(\tau+M)\overset{\eqref{eq:026}}{\underset{\eqref{eq:025}}=}f(\tau),\qquad\tau\in\calH, \]
	and $ \ref{enum:011:2} $ holds by the following estimate: since $ k>2 $,
	\begin{align*}
	\int_{\Gamma_\infty\backslash\calH}\norm{f(\tau)}\,\Im(\tau)^{\frac k2}\,d\mathrm v(\tau)
	&\overset{\eqref{eq:025}}=\int_0^M\int_0^{\infty}e^{-2\pi\frac{\nu+m_j}My}\,\norm{U^{-1}e_j}\,y^{\frac k2-2}\,dy\,dx\\
	&\overset{\phantom{\eqref{eq:025}}}=\frac{M^{\frac k2}}{\left(2\pi(\nu+m_j)\right)^{\frac k2-1}}\,\int_0^\infty e^{-y}\,y^{\frac k2-2}\,dy<\infty.
	\end{align*}
	
	\ref{enum:028:2} First, we note that the Fourier expansion \eqref{eq:032} exists by the same argument as the Fourier expansion \eqref{eq:033}. Now we have
	\[ \begin{aligned}
	&\scal F{\Psi_{k,\rho,\Gamma,\nu,U,j}}_{\calS(k,\rho,\Gamma)}\\
	&\quad\overset{\eqref{eq:031}}=\int_{\Gamma_\infty\backslash\calH}\scal{F(\tau)}{e^{2\pi i\frac{\nu+m_j}M\tau}\,U^{-1}e_j}_{\C^p}\,\Im(\tau)^k\,d\mathrm v(\tau)\\
	&\quad\overset{\phantom{\eqref{eq:032}}}=\int_{\Gamma_\infty\backslash\calH}\scal{UF(\tau)}{e^{2\pi i\frac{\nu+m_j}M\tau}e_j}_{\C^p}\,\Im(\tau)^k\,d\mathrm v(\tau)\\
	&\quad\overset{\eqref{eq:032}}=\lim_{R\to0+}\int_0^M\int_R^\infty\sum_{n=0}^\infty b_n(j)\,e^{2\pi i\frac{n-\nu}Mx}\,e^{-2\pi\frac{n+\nu+2m_j}My}\,y^{k-2}\,dy\,dx\\
	&\quad\overset{\phantom{\eqref{eq:032}}}=b_\nu(j)\,M\,\lim_{R\to0+}\int_R^\infty e^{-4\pi\frac{\nu+m_j}My}\,y^{k-2}\,dy\\
	&\quad\overset{\phantom{\eqref{eq:032}}}=b_\nu(j)\,\frac{M^k\,\Gamma(k-1)}{(4\pi(\nu+m_j))^{k-1}},
	\end{aligned} \]
	where the second equality holds because $ U $ is a unitary matrix, and the fourth one by the dominated convergence theorem.
\end{proof}

We note the following direct consequence of Proposition \ref{prop:030}.

\begin{corollary}
	Let $ k\in\R_{>2} $. Then,
	\[ \calS(k,\rho,\Gamma)=\mathrm{span}_\C\left\{\Psi_{k,\rho,\Gamma,\nu,U,j}:\nu\in\Z_{\geq0},\ j\in\left\{1,\ldots,p\right\}\right\}. \]
\end{corollary}

Finally, applying our non-vanishing criterion (Theorem \ref{thm:019}), we obtain the following result on the non-vanishing of VVMFs $ \Psi_{k,\rho,\Gamma,\nu,U,j} $.

\begin{theorem}
	Let $ k\in\R_{>2} $ and $ N\in\Z_{>0} $. Let $ \Gamma\in\left\{\Gamma_0(N),\left<-I_2\right>\Gamma_1(N),\left<-I_2\right>\Gamma(N)\right\} $ and
	\[ M:=\begin{cases}
	1,&\text{if }\Gamma\in\left\{\Gamma_0(N),\left<-I_2\right>\Gamma_1(N)\right\}\\
	N,&\text{if }\Gamma=\left<-I_2\right>\Gamma(N).
	\end{cases} \]
	Let $ \rho:\Gamma\to\GL_p(\C) $ be a normal unitary representation, and fix $ U\in\U(p) $ and $ m_1,\ldots,m_p\in\left]0,1\right]\cap\Q $ such that
	\[ \rho\left(T^M\right)=e^{-2\pi i\kappa M}\,U^{-1}\,\diag(e^{2\pi im_1},\ldots,e^{2\pi im_p})\,U. \]
	Then, $ \Psi_{k,\rho,\Gamma,\nu,U,j}\not\equiv0 $ if
	\begin{equation}\label{eq:034}
	\nu+m_j\leq\frac{MN}{4\pi}\left(k-\frac83\right).
	\end{equation}
\end{theorem}

\begin{proof}
	We apply Theorem \ref{thm:019}\ref{thm:019:1} with
	\[ A:=\left]0,M\right]\times\left]\frac1N,\infty\right[. \]
	
	Let us prove that the set $ A $ defined in this way satisfies \ref{thm:019:1:1}. Let $ \tau\in A $ and $ \gamma=\begin{pmatrix}a&b\\c&d\end{pmatrix}\in\Gamma $ such that $ \gamma.\tau\in A $. Then $ c=0 $, because otherwise we would have $ \abs c\geq N $ and consequently
	\[ \frac1N<\Im(\gamma.\tau)\overset{\eqref{eq:007}}=\frac{\Im(\tau)}{(c\Re(\tau)+d)^2+(c\Im(\tau))^2}\leq\frac{\Im(\tau)}{(c\Im(\tau))^2}=\frac1{c^2\Im(\tau)}<\frac1{N^2\cdot\frac1N}=\frac1N. \]
	Thus, $ \gamma\in\Gamma_\infty=\left<\pm T^M\right> $, hence $ \gamma.\tau=\tau+nM $ for some $ n\in\Z $. The fact that $ \Re(\tau),\Re(\gamma.\tau)\in\left]0,M\right] $ implies that $ n=0 $, hence $ \gamma.\tau=\tau $, which proves \ref{thm:019:1:1}.
	
	On the other hand, our set $ A $ satisfies \ref{thm:019:1:2} if and only if
	\[ \begin{aligned}
	\int_{\Gamma_\infty\backslash\Gamma_\infty.A}&\norm{e^{2\pi i\frac{\nu+m_j}M\tau}\,U^{-1}e_j}\,\Im(\tau)^{\frac k2}\,d\mathrm v(\tau)\\
	&>\int_{\Gamma_\infty\backslash(\Gamma_\infty.A)^c}\norm{e^{2\pi i\frac{\nu+m_j}M\tau}\,U^{-1}e_j}\,\Im(\tau)^{\frac k2}\,d\mathrm v(\tau),
	\end{aligned} \]
	i.e., recalling that $ U\in\U(p) $ and that $ \left]0,M\right]\times\left]0,\infty\right[ $ is a fundamental domain for $ \Gamma_\infty $ in $ \calH $, if and only if we have
	\[ \int_0^M\int_{\frac1N}^\infty e^{-2\pi\frac{\nu+m_j}My}\,y^{\frac k2-2}\,dy\,dx>\int_0^M\int_0^{\frac1N} e^{-2\pi\frac{\nu+m_j}My}\,y^{\frac k2-2}\,dy\,dx \]
	or equivalently
	\[ \int_{\frac{2\pi(\nu+m_j)}{MN}}^\infty t^{\frac k2-2}\, e^{-t}\,dt>\int_0^{\frac{2\pi(\nu+m_j)}{MN}}t^{\frac k2-2}\,e^{-t}\,dt, \]
	i.e., if and only if
	\[ \frac{2\pi(\nu+m_j)}{MN}<\M_{\Gamma\left(\frac k2-1,1\right)}, \]
	where $ \M_{\Gamma(a,b)}\in\R_{>0} $ is the median of the gamma distribution $ \Gamma(a,b) $ with parameters $ a,b\in\R_{>0} $, determined by the condition
	\[ \int_0^{\M_{\Gamma(a,b)}}x^{a-1}\,e^{-bx}\,dx=\int_{\M_{\Gamma(a,b)}}^\infty x^{a-1}\,e^{-bx}\,dx. \]
	Applying Chen and Rubin's estimate \cite[Theorem 1]{chen_rubin}, stating that
	\[ a-\frac13<M_{\Gamma(a,1)}<a,\qquad a\in\R_{>0}, \]
	it follows that \ref{thm:019:1:2} holds if
	\[ \frac{2\pi(\nu+m_j)}{MN}\leq\frac k2-\frac43, \]
	which is equivalent to \eqref{eq:034}.	
\end{proof}

\section{Elliptic vector-valued Poincar\'e series}\label{sec:007}

In this section, we use Proposition \ref{prop:010} and Theorem \ref{thm:019} to construct and study the non-vanishing of the vector-valued analogues of elliptic Poincar\'e series. The latter cusp forms were studied already by Petersson \cite[(8)]{petersson}.

Let $ \Gamma\ni-I_2 $ be a subgroup of finite index in $ \SL_2(\Z) $, and let $ \rho:\Gamma\to\GL_p(\C) $ be a normal unitary representation. 

\begin{proposition}\label{prop:035}
	Let $ k\in\R_{>2} $, $ \nu\in\Z_{\geq0} $, $ \xi\in\calH $. Then, we have the following: 
	\begin{enumerate}[label=\textup{(\arabic*)},leftmargin=*,align=left]
		\item\label{enum:035:1} Let $ u\in\C^p $. The Poincar\'e series
		\[ \Phi_{k,\rho,\Gamma,\nu,\xi,u}:=P_{\left<-I_2\right>\backslash\Gamma,\rho}\left(\frac{(\spacedcdot-\xi)^\nu}{\left(\spacedcdot-\overline\xi\right)^{\nu+k}}\,u\right) \]
		converges absolutely and uniformly on compact sets in $ \calH $ and defines an element of $ \calS(k,\rho,\Gamma) $.
		\item\label{enum:035:2} For every $ j\in\left\{1,\ldots,p\right\} $, we have
		\begin{equation}\label{eq:036}
		\scal F{\Phi_{k,\rho,\Gamma,\nu,\xi,e_j}}_{\calS(k,\rho,\Gamma)}=\frac{4\pi}{(4\Im(\xi))^k}\, \frac{\nu!}{(k-1)k\cdots(k+\nu-1)}\,b_{\nu,\xi}(j)
		\end{equation}
		for every $ F=(F_1,\ldots,F_p)\in\calS(k,\rho,\Gamma) $, where $ b_{\nu,\xi}(j)\in\C $ are coefficients in the expansion
		\begin{equation}\label{eq:037}
		\left(\tau-\overline\xi\right)^k\,F_j(\tau)=\sum_{n=0}^\infty b_{n,\xi}(j)\,\left(\frac{\tau-\xi}{\tau-\overline\xi}\right)^n,\qquad\tau\in\calH.
		\end{equation}
	\end{enumerate}
\end{proposition}

\begin{proof}
	\ref{enum:035:1} The claim follows from Proposition \ref{prop:010} as soon as we prove that $ f:\calH\to\C^p $,
	\[ f(\tau):=\frac{(\tau-\xi)^{\nu}}{\left(\tau-\overline\xi\right)^{\nu+k}}\,u, \]
	satisfies \ref{enum:011:2}. Applying the change of variables $ \tau\mapsto n_{\Re(\xi)}a_{\Im(\xi)}.\tau $, we obtain
	\[ \begin{aligned}
	\int_{\left<-I_2\right>\backslash\calH}\norm{f(\tau)}\,\Im(\tau)^{\frac k2}\,d\mathrm v(\tau)
	&=\frac{\norm{u}}{\Im(\xi)^{\frac k2}}\,\int_{\calH}\abs{\frac{\tau-i}{\tau+i}}^\nu\,\frac{\Im(\tau)^{\frac k2}}{\abs{\tau+i}^k}\,d\mathrm v(\tau)\\ 
	&\leq\frac{\norm u}{\Im(\xi)^{\frac k2}}\,\int_\calH\frac{\Im(\tau)^{\frac k2}}{\abs{\tau+i}^k}\,d\mathrm v(\tau)\\
	&=\frac{\norm u}{\Im(\xi)^{\frac k2}}\,\int_0^\infty\int_\R\frac{y^{\frac k2-2}}{\left(x^2+(y+1)^2\right)^{\frac k2}}\,dx\,dy\\
	&=\frac{\norm u}{\Im(\xi)^{\frac k2}}\,\int_\R\frac{dx}{\left(x^2+1\right)^{\frac k2}}\,\int_0^\infty\frac{y^{\frac k2-2}}{(y+1)^{k-1}}\,dy,
	\end{aligned} \]
	introducing the change of variables $ x\mapsto (y+1)x $ in the last equality. Since $ k>2 $, the right-hand side is finite, which proves \ref{enum:011:2}.
	
	\ref{enum:035:2} Let $ F\in\calS(k,\rho,\Gamma) $. By Lemma \ref{lem:037}, we have
	\[ \scal F{\Phi_{k,\rho,\Gamma,\nu,\xi,e_j}}_{\calS(k,\rho,\Gamma)}
	=\int_\calH\scal{F(\tau)}{\frac{(\tau-\xi)^\nu}{\left(\tau-\overline\xi\right)^{\nu+k}}\,e_j}_{\C^p}\,\Im(\tau)^k\,d\mathrm v(\tau). \]
	Using \eqref{eq:037} and introducing the change of variables $ \tau\mapsto n_{\Re(\xi)}a_{\Im(\xi)}.\tau $, we see that the right-hand side equals
	\[ \Im(\xi)^{-k}\,\int_\calH\sum_{n=0}^\infty b_{n,\xi}(j)\,\left(\frac{\tau-i}{\tau+i}\right)^{n-\nu}\,\abs{\frac{\tau-i}{\tau+i}}^{2\nu}\,\frac{\Im(\tau)^k}{\abs{\tau+i}^{2k}}\,d\mathrm v(\tau), \]
	which, introducing the substitution $ w=\frac{\tau-i}{\tau+i} $ (so $ dx\,dy=\frac4{\abs{1-w}^4}dw $) and denoting $ \calD:=\left\{w\in\C:\abs w<1\right\} $, equals
	\[ \frac4{(4\Im(\xi))^k}\,\int_\calD\sum_{n=0}^\infty b_{n,\xi}(j)\,w^{n-\nu}\,\abs w^{2\nu}\left(1-\abs w^2\right)^{k-2}\,dw. \]
	Going over to polar coordinates, we obtain
	\[ \frac4{(4\Im(\xi))^k}\,\lim_{R\nearrow1}\int_0^R\int_0^{2\pi}\sum_{n=0}^\infty b_{n,\xi}(j)\,r^{n+\nu+1}\,\left(1-r^2\right)^{k-2}\,e^{i(n-\nu)t}\,dt\,dr, \]
	i.e., applying the dominated convergence theorem,
	\[ \begin{aligned}
	\frac4{(4\Im(\xi))^k}&\,\lim_{R\nearrow1}\sum_{n=0}^\infty b_{n,\xi}(j)\,\int_0^Rr^{n+\nu+1}\,\left(1-r^2\right)^{k-2}\,dr\,\int_0^{2\pi}e^{i(n-\nu)t}\,dt\\
	&=\frac{8\pi}{(4\Im(\xi))^k}\,b_{\nu,\xi}(j)\,\int_0^1r^{2\nu+1}\,\left(1-r^2\right)^{k-2}\,dr\\
	&=\frac{4\pi}{(4\Im(\xi))^k}\,b_{\nu,\xi}(j)\,\frac{\nu!}{(k-1)k\cdots(k+\nu-1)},
	\end{aligned} \]	
	where the last equality is obtained by $ \nu $-fold partial integration after substituting $ t=r^2 $.
\end{proof}

As a direct consequence of Proposition \ref{prop:035}, we obtain the following corollary.

\begin{corollary}
	Let $ k\in\R_{>2} $ and $ \xi\in\calH $. Then,
	\[ \calS(k,\rho,\Gamma)=\mathrm{span}_\C\left\{\Phi_{k,\rho,\Gamma,\nu,\xi,e_j}:\nu\in\Z_{\geq0},\ j\in\left\{1,\ldots,p\right\}\right\}. \]
\end{corollary}

In the following theorem, we give a result on the non-vanishing of elliptic VVPSs. To state it, we need the notion of the median $ \M_{B(a,b)}\in\left]0,1\right[ $ of the beta distribution $ B(a,b) $ with parameters $ a,b\in\R_{>0} $, defined by the condition
\[ \int_0^{\M_{B(a,b)}}x^{a-1}\,(1-x)^{b-1}\,dx=\int_{\M_{B(a,b)}}^1x^{a-1}\,(1-x)^{b-1}\,dx. \]

\begin{theorem}
	Let $ N\in\Z_{\geq2} $, and let $ \Gamma\ni-I_2 $ be a subgroup of finite index in $ \left<-I_2\right>\Gamma(N) $. Let $ \rho:\Gamma\to\GL_p(\C) $ be a normal unitary representation. Let $ k\in\R_{>2} $, $ \nu\in\Z_{\geq0} $, and $ u\in\C^p\setminus\left\{0\right\} $. If
	\begin{equation}\label{eq:042}
	N>\frac{4\left({\M_{B\left(\frac\nu2+1,\frac k2-1\right)}}\right)^{\frac12}}{1-\M_{B\left(\frac\nu2+1,\frac k2-1\right)}},
	\end{equation}
	then 
	\[ \Phi_{k,\rho,\Gamma,\nu,i,u}\not\equiv0. \]
\end{theorem}

\begin{proof}
	We recall that $ \Phi_{k,\rho,\Gamma,\nu,i,u}=P_{\left<-I_2\right>\backslash\Gamma,\rho}f $ for $ f:\calH\to\C^p $,
	\begin{equation}\label{eq:038}
	f(\tau):=\frac{(\tau-i)^\nu}{(\tau+i)^{\nu+k}}\,u. 
	\end{equation}
	To apply Theorem \ref{thm:019}, it suffices to find a Borel-measurable set $ C\subseteq\SL_2(\R) $ satisfying \ref{thm:019:2:1}--\ref{thm:019:2:3} with $ \Lambda=\left<-I_2\right> $. Following the idea of \cite[Lemma 6-5]{muicJNT}, let us look for such a set $ C $ of the form
	\[ C_r:=K\left\{h_t:t\in\left[0,r\right]\right\}K \]
	with $ r\in\R_{>0} $. For every $ r\in\R_{>0} $, the set $ C_r $ obviously satisfies \ref{thm:019:2:1}. Next, by \cite[Lemma 6-20]{muicJNT} we have
	\[ \max_{g\in C_rC_r^{-1}}\norm g=\sqrt{2\cosh(4r)}, \]
	and on the other hand, obviously 
	\[ \min_{\gamma\in\Gamma\setminus\left<-I_2\right>}\norm\gamma\geq\sqrt{N^2+2}, \]
	so $ C_r $ satisfies \ref{thm:019:2:2} if 
	\begin{equation}\label{eq:040}
	\sqrt{2\cosh(4r)}<\sqrt{N^2+2}.
	\end{equation}
	Next, one checks easily that (for every $ f:\calH\to\C^p $)
	\[ \begin{aligned}
	&\norm{f(\kappa_{\theta_1}h_t\kappa_{\theta_2}.i)}\,\abs{j(\kappa_{\theta_1}h_t\kappa_{\theta_2},i)}^{-k}\\
	&\qquad=\norm{f\left(\frac{e^t\,i\,cos\theta_1-e^{-t}\,\sin\theta_1}{e^t\,i\,\sin\theta_1+e^{-t}\cos\theta_1}\right)}\,\abs{e^t\,i\,\sin\theta_1+e^{-t}\,\cos\theta_1}^{-k}
	\end{aligned} \]
	for all $ t\in\R_{\geq0} $ and $ \theta_1,\theta_2\in\R $. From this, it follows by an elementary computation, using \eqref{eq:038}, that 
	\[ \norm{f(\kappa_{\theta_1}h_t\kappa_{\theta_2}.i)}\,\abs{j(\kappa_{\theta_1}h_t\kappa_{\theta_2},i)}^{-k}=\frac{\tanh^\nu t}{(2\cosh t)^k}\,\norm u \]
	for all $ t\in\R_{\geq0} $ and $ \theta_1,\theta_2\in\R $.
	Thus, using \eqref{eq:039}, $ C_r $ satisfies \ref{thm:019:2:3} if and only if
	\begin{equation}\label{eq:041}
	\int_0^r\frac{\tanh^\nu t}{\cosh^kt}\,\sinh(2t)\,dt>\int_r^\infty\frac{\tanh^\nu t}{\cosh^kt}\,\sinh(2t)\,dt.
	\end{equation}
	The computation from \cite[proof of Proposition 6.7]{zunarManu} shows that there exists $ r\in\R_{>0} $ satisfying both \eqref{eq:040} and \eqref{eq:041} if and only if \eqref{eq:042} holds. This proves the theorem.
\end{proof}

\begin{remark}
	For concrete values of $ \nu $ and $ k $, it is easy to compute the value of the right-hand side in \eqref{eq:042} explicitly using mathematical software (e.g., in R 3.3.2, $ \M_{B(a,b)} $ is implemented as \texttt{qbeta(0.5,a,b)}). Moreover, \cite[Corollary 6.18]{zunarManu} lists a few elementary sufficient conditions on $ \nu $, $ k $ and $ N $ for the inequality \eqref{eq:042} to hold.
\end{remark}

\bibliographystyle{amsplain}

\begin{thebibliography}{999999}	
	
	\bibitem{bantay1} Bantay, P.: The dimension of spaces of vector-valued modular forms of integer weight. Lett.\ Math.\ Phys.\ 103(11), 1243--1260 (2013)
	
	\bibitem{bantay2} Bantay, P.: A trace formula for vector-valued modular forms. Commun. Contemp. Math.\ 17(6) (2015)
	
	\bibitem{bantay_gannon} Bantay, P., Gannon, T.: Vector-valued modular functions for the modular	group and the hypergeometric equation. Commun.\ Number Theory Phys.\ 1(4), 651--680 (2007)

	\bibitem{borcherds} Borcherds, R.\ E.: Automorphic forms with singularities on Grassmannians. Invent.\ Math.\ 132(3), 491--562 (1998)
	
	\bibitem{borcherds_gross} Borcherds, R.\ E.: The Gross-Kohnen-Zagier theorem in higher dimensions. Duke Math.\ J.\ 97(2), 219--233 (1999) 

	\bibitem{candelori_franc} Candelori, L., Franc, C.: Vector-valued modular forms and the modular orbifold of elliptic curves. Int.\ J.\ Number Theory 13(1), 39--63 (2017)
	
	\bibitem{chen_rubin} Chen, J., Rubin, H.: Bounds for the difference between median and mean of gamma and Poisson distributions. Statist.\ Probab.\ Lett.~4(6), 281--283 (1986)
	
	\bibitem{dong_li_mason} Dong, C., Li, H., Mason, G.: Modular-invariance of trace functions in orbifold theory and generalized Moonshine. Comm.\ Math.\ Phys.\ 214, 1--56 (2000)

	\bibitem{eichler_zagier} Eichler, M., Zagier, D.: The theory of Jacobi forms. Progress in Mathematics 55, Birkh\"auser Boston, Inc., Boston, MA (1985) 

	\bibitem{franc_mason} Franc, C., Mason, G.: Fourier coefficients of vector-valued modular forms of dimension 2. Canad.\ Math.\ Bull.\ 57(3), 485--494 (2014)

	\bibitem{franc_mason2} Franc, C., Mason, G.: On the structure of modules of vector-valued modular forms. Ramanujan J.\ 47(1), 117--139 (2018)
	
	\bibitem{gannon} Gannon T.: The theory of vector-valued modular forms for the modular group. Conformal field theory, automorphic forms and related topics, Contrib.\ Math.\ Comput.\ Sci.\ 8, Springer, Heidelberg, 247--286 (2014)

	\bibitem{grobner} Grobner, H.: Smooth automorphic forms and smooth automorphic representations. Book in preparation, to appear in Series on Number Theory and Its Applications, WorldScientific.
	
	\bibitem{km_generalized} Knopp, M., Mason, G.: Generalized modular forms. J.\ Number Theory 99(1), 1--28 (2003)

	\bibitem{km_fourier} Knopp, M., Mason, G.: On vector-valued modular forms and their Fourier coefficients. Acta Arith.~110(2), 117--124 (2003)

	\bibitem{km_poincare} Knopp, M., Mason, G.: Vector-valued modular forms and Poincar\'e series. Illinois J.\ Math.\ 48(4), 1345--1366 (2004)

	\bibitem{kohnen} Kohnen, W.: Nonvanishing of Hecke L-functions associated to cusp forms inside the critical strip. J.\ Number Theory 67(2), 182--189 (1997)

	\bibitem{lehner} Lehner, J.: On the non-vanishing of Poincar\'e series.	Proc.\ Edinb.\ Math.\ Soc.\ (2) 23(2), 225--228 (1980)

	\bibitem{marks} Marks, C.: Irreducible vector-valued modular forms of dimension less than six. Illinois J.\ Math.\ 55(4), 1267--1297 (2011)

	\bibitem{marks_mason} Marks, C., Mason, G.: Structure of the module of vector-valued modular forms. J.\ Lond.\ Math.\ Soc.\ (2) 82(1), 32--48 (2010)
	
	\bibitem{mason_diff} Mason, G.: Vector-valued modular forms and linear differential operators. Int.\ J.\ Number Theory 3(3), 377--390 (2007)
	
	\bibitem{mason_2dim} Mason, G.: 2-dimensional vector-valued modular forms. Ramanujan J.\ 17(3), 405--427 (2008)
	
	\bibitem{milas}  Milas, A.: Virasoro algebra, Dedekind $ \eta $-function and specialized Macdonald identities. Transform.\ Groups 9(3), 273--288 (2004)
	
	\bibitem{miyake} Miyake, T.: Modular forms. Translated from the 1976 Japanese original by Yoshitaka Maeda. Reprint of the first 1989 English edition. Springer Monographs in Mathematics. Springer-Verlag, Berlin (2006)
	
	\bibitem{miyamoto} Miyamoto, M.: Modular invariance of vertex operator algebras satisfying $ C_2 $-cofiniteness. Duke Math.\ J.\ 122(1), 51--91 (2004)
	
	\bibitem{muicMathAnn} Mui\'c, G.: On a construction of certain classes of cuspidal automorphic forms via Poincar\'e series. Math.\ Ann.\ 343(1), 207--227 (2009)
	
	\bibitem{muicJNT} Mui\' c, G.: On the cuspidal modular forms for the Fuchsian groups of the first kind. J.\ Number Theory 130(7), 1488--1511 (2010)
	
	\bibitem{muicIJNT} Mui\' c, G.: On the non-vanishing of certain modular forms. Int.\ J.\ Number Theory 7(2), 351--370 (2011)
	
	\bibitem{muicLFunk} Mui\' c, G.: On the analytic continuation and non-vanishing of L-functions. Int.\ J.\ Number Theory 8(8), 1831--1854 (2012)
	
	\bibitem{petersson} Petersson, H.: Einheitliche Begr\"undung der Vollst\"andigkeitss\"atze f\"ur die Poincar\'eschen Reihen von reeller Dimension bei beliebigen Grenzkreisgruppen von erster Art. Abh.\ Math.\ Sem.\ Hansischen Univ.~14, 22--60 (1941)
	
	\bibitem{poincare} Poincar\'e, H.: M\'emoire sur les fonctions fuchsiennes. Acta Math.\ 1(1), 193--294 (1882)
	
	\bibitem{rankin} Rankin, R.\ A.: The vanishing of Poincar\'e series. Proc.\ Edinb.\ Math.\ Soc.\ (2) 23(2), 151--161 (1980)
	
	\bibitem{saber_sebbar} Saber, H., Sebbar, A.: On the existence of vector-valued automorphic forms. Kyushu J.\ Math.\ 71(2), 271--285 (2017) 
	
	\bibitem{selberg} Selberg, A.: On the estimation of Fourier coefficients of modular forms. Theory of Numbers. Proc.\ Sympos.\ Pure Math.\ VIII, Providence, RI, Amer.\ Math.\ Soc., 1--15 (1965)
	
	\bibitem{zhu} Zhu, Y.: Modular invariance of characters of vertex operator algebras. J.\ Amer.\ Math.\ Soc.\ 9(1), 237--302 (1996)
	
	\bibitem{zunarGlas} \v Zunar, S.: On Poincar\'e series of half-integral weight,  Glas.\ Mat.\ Ser.\ III 53(2), 239--264 (2018)
	
	\bibitem{zunarManu} \v Zunar, S.: On the non-vanishing of Poincar\'e series on the metaplectic group. Manuscripta Math.\ 158(1--2), 1--19 (2019)

	\bibitem{zunarRama} \v Zunar, S.: On the non-vanishing of L-functions associated to cusp forms of half-integral weight. Ramanujan J.\ 51(3), 455--477 (2020)
	
\end{thebibliography}

\end{document}